\documentclass[11pt,a4paper]{article}

\usepackage{epsf,epsfig,amsfonts,amsgen,amsmath,amstext,amsbsy,amsopn,amsthm
}
\usepackage{ebezier,eepic}
\usepackage{color}

\newtheorem{theorem}{Theorem}

\newtheorem{lemma}{Lemma}
\newtheorem{false statement}{False statement}

\theoremstyle{definition}

\newtheorem{claim}{Claim}

\newtheorem{conjecture}{Conjecture}
\newtheorem{corollary}[theorem]{Corollary}
\newtheorem{problem}{Problem}

\newcounter{mathitem}
  {\begin{list}{{$(\roman{mathitem})$}}{
   \setcounter{mathitem}{0}
   \usecounter{mathitem}
   \setlength{\topsep}{0pt plus 2pt minus 0pt}
   \setlength{\parskip}{0pt plus 2pt minus 0pt}
   \setlength{\partopsep}{0pt plus 2pt minus 0pt}
   \setlength{\parsep}{0pt plus 2pt minus 0pt}
   \setlength{\leftmargin}{35pt}
   \setlength{\itemsep}{0pt plus 2pt minus 0pt}}}
  {\end{list}}

\date{\dateline{April 12, 2013}{}}
\begin{document}

\title{\bf\Large Notes on a conjecture of Manoussakis concerning Hamilton cycles in digraphs}

\date{}

\author{Bo Ning\thanks{Corresponding author. Email address: ningbo\_math84@mail.nwpu.edu.cn}\\
\small Department of Applied Mathematics, School of Science,\\
\small  Northwestern Polytechnical University, Xi'an, Shaanxi 710072, P.R.~China\\[2mm]}
\maketitle
\begin{abstract}
In 1992, Manoussakis conjectured that a strongly 2-connected digraph $D$ on $n$ vertices is hamiltonian if for every two distinct pairs of independent vertices $x,y$ and $w,z$ we have $d(x)+d(y)+d(w)+d(z)\geq 4n-3$. In this note we show that $D$ has a Hamilton path, which gives an affirmative evidence supporting this conjecture.


\smallskip

\end{abstract}
\medskip
\section{Introduction}
In this note, we consider simple digraphs only. For convenience of the reader, we provide all necessary terminology and notation in one section, Section 2. For those not defined here, we refer the reader to \cite{Bang_Gutin}.

A basic topic in digraph theory is that of finding degree conditions for a digraph to be hamiltonian. In particular, Ghouila-Houri \cite{Ghouila_Houri} proved a fundamental theorem which states that every strongly connected digraph on $n$ vertices is hamiltonian if the degree of every vertex is at least $n$.

\begin{theorem}[\rm Ghouila-Houri \cite{Ghouila_Houri}]
Let $D$ be a strongly connected digraph on $n$ vertices. If $d(x)\geq n$ for any vertex $x\in V$, then $D$ is hamiltonian.
\end{theorem}

Woodall \cite{Woodall} proved the following result, which improved Ghouila-Houri's theorem.

\begin{theorem}[\rm Woodall \cite{Woodall}]
Let $D$ be a digraph on $n$ vertices. If $d^+(x)+d^-(y)\geq n$ for any pair of vertices $x$ and $y$ such that
$xy\notin A(D)$, then $D$ is hamiltonian.
\end{theorem}

Meyniel \cite{Meyniel} generalized both theorems of Ghoulia-Houri and Woodall. Bondy and Thomassen \cite{Bondy_Thomassen} gave a new proof of Meyniel's theorem by proving a slightly stronger result. For another proof of Meyniel's theorem, see \cite{Overbeck}.

\begin{theorem}[\rm Meyniel \cite{Meyniel}]
Let $D$ be a strongly connected digraph on $n$ vertices. If $d(x)+d(y)\geq 2n-1$ for any pair of nonadjacent vertices in $D$, then $D$ is hamiltonian.
\end{theorem}

Manoussakis \cite{Manoussakis} gave another generalization of Woodall's theorem as follows.

\begin{theorem}[\rm Manoussakis \cite{Manoussakis}]
Let $D$ be a strongly connected digraph on $n$ vertices. For any triple of vertices $x,y,z\in V$, where $x$ is nonadjacent to $y$, if there hold $d(x)+d(y)+d^+(x)+d^-(z)\geq 3n-2$ (if $xz\notin A$) and $d(x)+d(y)+d^+(z)+d^-(x)\geq 3n-2$ (if $zx\notin A$), then $D$ is hamiltonian.
\end{theorem}

Manoussakis \cite{Manoussakis} proposed the following conjecture. If this conjecture is true, then it can be seen as an extension of Theorem 4.

\begin{conjecture}[\rm Manoussakis \cite{Manoussakis}]
Let $D$ be a strongly 2-connected digraph such that for all distinct pairs of nonadjacent vertices $x,y$ and $w,z$ we have $d(x)+d(y)+d(w)+d(z)\geq 4n-3$. Then $D$ is hamiltonian.
\end{conjecture}

Manoussakis \cite{Manoussakis} gave an example to show that Conjecture 1 is almost best. Here we gave another example. Let $D$ be an associated digraph of $K_{\frac{n-1}{2},\frac{n+1}{2}}$, where $n\geq 9$ is odd. Let $X,Y$ be two parts of $D$ such that $|X|=\frac{n-1}{2},|Y|=\frac{n+1}{2}$. Then the degree sum of any four vertices in $X$ is $4(n+1)$ and the degree sum of any four vertices in $Y$ is $4(n-1)$. Furthermore, we can see the degree sum of all distinct pairs of nonadjacent vertices in $D$ is at least $4n-4$ and $D$ is not hamiltonian.

To our knowledge, there are no further references on this conjecture. In this note we prove the following result, and it may be a first step towards confirming Conjecture 1.

\begin{theorem}
Let $D$ be a strongly 2-connected digraph such that for all distinct pairs of nonadjacent vertices $x,y$ and $w,z$ we have $d(x)+d(y)+d(w)+d(z)\geq 4n-3$. Then $D$ has a longest cycle of length at least $n-1$.
\end{theorem}

The following result is a direct corollary.

\begin{corollary}
Let $D$ be a strongly 2-connected digraph such that for all distinct pairs of nonadjacent vertices $x,y$ and $w,z$ we have $d(x)+d(y)+d(w)+d(z)\geq 4n-3$. Then $D$ has a Hamilton path.
\end{corollary}

\section{Terminology and notation}
In this section, we will give necessary notation and terminology. Throughout this note, we use $D$ to denote a digraph (directed graph), and $V(D)$ and $A(D)$ to denote the \emph{vertex set} and \emph{arc set} of $D$, respectively. When there is no danger of ambiguity, we use $V$ and $A$ instead of $V(D)$ and $A(D)$, respectively. For an arc $xy\in A$, $x$ is always referred to as the \emph{origin}, and $y$, as the \emph{terminus}. Throughout this note, simple digraphs are just considered, that is, digraphs with no two arcs with the same origin and terminus, and no loops (an arc with the same vertex as the origin and terminus meantime).

We say that $D$ is \emph{strongly $k$-connected} if for any ordered pair of vertices $\{u,v\}$, there are $k$ internally disjoint directed paths from $u$ to $v$. For two vertices $u,v\in V$, we say that $u$ \emph{dominates} (is \emph{dominated} by) $v$ if there is an arc $uv\in A$ ($vu\in A$), and $u,v$ are called a pair of \emph{nonadjacent vertices} if $uv\notin A$ and $vu\notin A$. For a vertex $v$ and a subdigraph $H$ of $D$, the \emph{out-neighbor set} (\emph{in-neighbor set}) of $v$ in $H$, denoted by $N^+_{H}(v)$ ($N^-_{H}(v)$), is the set of those vertices in $H$ dominated by (dominating) $v$. The \emph{out-degree} (\emph{in-degree}, \emph{degree}) of $v$ in $H$, denoted by $d^+_{H}(v)$ ($d^-_{H}(v)$, $d_H(v)$), equals $|N^+_{H}(v)| (|N^-_{H}(v)|, |N^+_{H}(v)|+|N^-_{H}(v)|)$. If there is no danger of ambiguity, then we use $d^+(v)$, $d^-(v)$ and $d(v)$ instead of $d^+_{D}(v)$, $d^-_{D}(v)$ and $d_{D}(v)$, respectively. We use $D-H$ to denote the subdigraph of $D$ induced by the vertex set $V(D)\backslash V(H)$.

A digraph $D$ on $n$ vertices is called \emph{hamiltonian} if there is a directed cycle of length $n$, and called \emph{pancyclic} if there are directed cycles with lengths from $2$ to $n$. Let $C$ be a directed cycle in $D$ with a given orientation. Let $u\in V(C)$. We use $u^-$ and $u^+$ to denote the \emph{predecessor} and \emph{successor} of $u$ along the orientation of $C$, respectively. For two vertices $u,v\in V(C)$, we use $C[u,v]$ to denote the segment from $u$ to $v$ along the orientation of $C$, and let $C(u,v)=C[u^+,v^-]$.

We also use some terminology and notation from \cite{Berman_Liu,Manoussakis}. Let $P=v_1v_2\cdots v_p$ be a path and $u$ be a vertex not on $P$. If there are two vertices $v_m$ and $v_{m+1}$ (where $m,m+1\in \{1,2,\ldots,p\}$) such that $v_mu\in A$ and $uv_{m+1}\in A$, then $P$ can be extended to include $u$ by replacing the arc $v_mv_{m+1}$ by the path $v_muv_{m+1}$. In this case, following \cite{Berman_Liu}, we say that \emph{$u$ can be inserted into $P$}. Let $D$ be a non-hamiltonian digraph on $n$ vertices and $C=x_1x_2\ldots x_kx_1$ be a longest cycle in $D$. Following \cite{Manoussakis}, we define a $C$-path of $D$ (with respect to a component $H$ of $D-C$) to be a path $P=x_py_1y_2\ldots y_tx_{p+\lambda}$, where $t\geq 1$, $x_p,x_{p+\lambda}$ are two distinct vertices of $C$, $\{y_1,\ldots,y_t\}\subset V(H)$, and $\lambda$ is chosen as the minimal one, that is, there is no path $P'=x_{p'}y'_1y'_2\ldots y'_{t'}x_{p'+\lambda'}$ such that $\lambda'\geq 1$, $0<\lambda'<\lambda$, $\{x_{p'},x_{p'+\lambda'}\}\subset \{x_p,x_{p+1},\ldots,x_{p+\lambda}\}$ (the subscripts of all the $x_i$'s are taken modulo $k$), where $\{y'_1,y'_2,\ldots, y'_{t'}\}\subset V(H)$.

\section{Proof of Theorem 5}
The following three lemmas are useful for our proof. The second lemma is a refinement of Lemma 2.3 in \cite{Manoussakis}.
\begin{lemma}[Bondy and Thomassen \cite{Bondy_Thomassen}]
Let $D$ be a digraph, $P$ be a directed path of $D$ and $v\in V(D)\backslash V(P)$. If $v$ can not be inserted into $P$, then
$d_P(v)\leq |P|+1$.
\end{lemma}
\begin{lemma}
Let $D$ be a non-hamiltonian digraph on $n$ vertices, $C=x_1x_2\ldots x_k$ be a longest cycle of $D$, $P=x_py_1y_2\ldots y_tx_{p+\lambda}$ be a $C$-path of $D$ (with respect to a component $H$ of $D-C$), $R=\{x_{p+1},x_{p+2},\ldots,x_{p+\lambda-1}\}$, and $S=\{v:v\in R,v~can~not~be~inserted~into~C[x_{p+\lambda},x_p]\}$. Then for any $y_i$, $i\in \{1,2,\ldots,t\}$, $s\in S$, $d(y_i)+d(s)\leq 2n-2$.
\end{lemma}

\begin{proof}
Since $C$ is longest in $D$, $y_i$ can not be inserted into $C[x_{p+\lambda},x_p]$. By Lemma 1,
\begin{align}
d_{C[x_{p+\lambda},x_p]}(y_i)\leq |C[x_{p+\lambda},x_p]|+1.
\end{align}
Since $s$ can not be inserted into $C[x_{p+\lambda},x_p]$, by Lemma 1,
\begin{align}
d_{C[x_{p+\lambda},x_p]}(s)\leq |C[x_{p+\lambda},x_p]|+1.
\end{align}
Since $P$ is a $C$-path of $D$, $y_i$ is nonadjacent to any vertex of $C[x_{p+1},x_{p+\lambda-1}]$. It follows that
\begin{align}
d_{C[x_{p+1},x_{p+\lambda-1}]}(y_i)=0.
\end{align}
Furthermore, we have
\begin{align}
d_{C[x_{p+1},x_{p+\lambda-1}]}(s)\leq 2(|C[x_{p+1},x_{p+\lambda-1}]|-1).
\end{align}
Let $H=D-C$. Moreover, $D$ has neither a directed path $y_iws$ nor a directed path $swy_i$, where $w\in V(H)\backslash \{y_i\}$,
since otherwise there is a $C$-path either from $x_p$ to $s$ or from $s$ to $x_{p+\lambda}$, and it contradicts the minimality of $\lambda$. This implies that
\begin{align}
d_{H}(y_i)+d_{H}(s)\leq 2(|H|-1).
\end{align}
By adding the inequalities (1)-(5), we have that
$d(y_i)+d(s)=d_{C[x_{p+\lambda},x_p]}(y_i)+d_{C[x_{p+\lambda},x_p]}(s)+d_{C[x_{p+1},x_{p+\lambda-1}]}(y_i)+d_{C[x_{p+1},x_{p+\lambda-1}]}(s)+d_{H}(y_i)+d_{H}(s)
\leq |C[x_{p+\lambda},x_p]|+1+|C[x_{p+\lambda},x_p]|+1+2(|C[x_{p+1},x_{p+\lambda-1}]|-1)+2(|H|-1)=2n-2$.

The proof is complete.
\end{proof}
\begin{lemma}[Berman and Liu \cite{Berman_Liu}]
Let $P$ and $Q$ be two (vertex) disjoint paths and $K$ be a subset of $V(P)$. If every vertex $z$ in $K$ can be inserted into $Q$,
then there exists a path $Q'$ with the same endpoints as $Q$ such that $V(Q)\subset V(Q')\subset V(Q)\cup V(P)$ and $Q'$ contains all
vertices of $K$.
\end{lemma}

\noindent{}
{\bf Proof of Theorem 5.} Suppose that $D$ is not hamiltonian. Let $C=x_1x_2\ldots x_k$ be a longest cycle in $D$ with a given orientation. Since $D$ is not hamiltonian, $k\leq n-1$ and $V(D)\backslash V(C)\neq \emptyset$. Let $H$ be a component of $D-C$ and $R=D-C-H$. Since $D$ is strongly 2-connected, there are at least two in-neighbors and two out-neighbors of $H$ in $C$. Thus there is a $C$-path (with respect to $H$), denote by $P=x_py_1y_2\ldots y_tx_{p+\lambda}$, where $t\geq 1$ and $x_p,x_{p+\lambda}\in V(C)$. Let $S=\{v:v\in C[x_{p+1},x_{p+\lambda-1}], v~can~not~be~inserted~into~C[x_{p+\lambda},x_p]\}$. Since $C$ is longest, $S\neq \emptyset$. Let $s$ be an arbitrary vertex of $S$.
By Lemma 2, we have
\begin{claim}
$d(y_i)+d(s)\leq 2n-2$ for $i\in \{1,2,\ldots,t\}$.
\end{claim}

The next claim can be easily deduced from the assumption of Theorem 5.
\begin{claim}
For any triple of distinct vertices $x,y,z$ such that $x,y$ and $x,z$ are two pairs of nonadjacent vertices, $2d(x)+d(y)+d(z)\geq 4n-3$.
\end{claim}

\begin{claim}$S=\{s\}$ and $t=1$.
\end{claim}

\begin{proof}
Assume that $|S|\geq 2$. Let $s,s'\in S$. Then by Claim 1, $d(y_1)+d(s)\leq 2n-2$ and $d(y_1)+d(s')\leq 2n-2$. By the choice of $P$, $y_1,s$ and $y_1,s'$ are two pairs of nonadjacent vertices, and we get $2d(y_1)+d(s)+d(s')\leq 4n-4$. By Claim 2, we get a contradiction. Hence $|S|=1$.

Assume that $t\geq 2$. By Claim 1, $d(s)+d(y_1)\leq 2n-2$ and $d(s)+d(y_2)\leq 2n-2$. By the choice of $P$, $s,y_1$ and $s,y_2$ are two pairs of nonadjacent vertices. Thus we obtain $2d(s)+d(y_1)+d(y_2)\leq 4n-4$, a contradiction by Claim 2. Hence $t=1$.
\end{proof}

\begin{claim}
$R=\emptyset$.
\end{claim}
\begin{proof}
Assume that $R\neq \emptyset$. Let $H'$ be a component of $R$. Since $D$ is strongly 2-connected, there is a $C$-path (with respect to $H'$), denoted $P'=x_{q}z_1\ldots z_{t'}x_{q+r}$, where $x_q,x_{q+r}\in V(C)$, $\{z_1,z_2,\ldots,z_{t'}\}\subseteq V(H')$ and the subscripts are taken modulo $k$. If every vertex of $C[x_{q+1},x_{q+r-1}]$ can be inserted into $C[x_{q+r},x_q]$, then by Lemma 3, there is a cycle longer than $C$, a contradiction. Thus there exists at least one vertex in $C[x_{q+1},x_{q+r-1}]$, say $s'$, such that it can not be inserted into $C[x_{q+r},x_q]$. By Lemma 2, we have $d(z_1)+d(s')\leq 2n-2$. Note that $z_1\in V(H')$ is a vertex different from $y_1$, and $s,y_1$ and $s',z_1$ are two distinct pairs of nonadjacent vertices. Thus we obtain $d(y_1)+d(s)+d(z_1)+d(s')\leq 4n-4$, a contradiction by the assumption of Theorem 5.
\end{proof}

\begin{claim}
$H=\{y_1\}$.
\end{claim}
\begin{proof}
Assume not. Then $H\backslash \{y_1\}\neq \emptyset$. Consider the digraph $D'=D-y_1$. Since $D$ is strongly 2-connected, $D'$ is strongly connected, and thus there is a directed path $P_0$ from $C$ to a component of $H-y_1$, say $H'$. W.l.o.g., let $x_i\in V(C)$ and $y\in H'$ be two vertices such that $x_iy\in A(D')$. Since $D'$ is strongly connected, there is also a directed path from $y$ to $x_i$, say $P'$.

Assume that there is no $C$-path in $D'$. We will show that $d_{D'}({x_i}^-)+d_{D'}(y)\leq 2|V(D')|-2$. First, we have observations that $N^+_{H'-P'}({x_i}^-)\cap N^-_{H'-P'}(y)=\emptyset$ (since otherwise there is a cycle longer than $C$ in $D'$), $N^-_{H'-P'}({x_i}^-)\cap N^+_{H'-P'}(y)=\emptyset$ (since otherwise there exists a $C$-path). It follows that
\begin{align*}
d_{H'-P'}({x_i}^-)+d_{H'-P'}(y)
&=d^+_{H'-P'}({x_i}^-)+d^-_{H'-P'}(y)+d^-_{H'-P'}({x_i}^-)+d^+_{H'-P'}(y)\\
&\leq 2|V(H')\backslash V(P')|
\end{align*}
It is obvious that no vertex in $P'\backslash \{x_i\}$ is a neighbor of ${x_i}^-$, that is, $d_{P'\backslash \{x_{i}\}}({x_i}^-)=0$, and $d_{P'\backslash \{x_{i}\}}(y)\leq 2|V(P'\backslash \{x_{i}\})|-2$. Furthermore, $d_{D'-C-H'}({x_i}^-)\leq 2|V(D')\backslash (V(C)\cup V(H'))|$ and $d_{D'-C-H'}(y)=0$. From the above facts, we obtain $d_{D'-C}({x_i}^-)+d_{D'-C}(y)\leq 2|V(D')\backslash V(C)|-2$. Note that $y$ is nonadjacent to any vertex of $C$ except for $x_i$. It follows that $d_C(y)\leq 2$. On the other hand, $d_{C}({x_i}^-)\leq 2(|C|-1)$. Together with these inequalities, we have $d_{D'}({x_i}^-)+d_{D'}(y)\leq 2|V(D')|-2$. Furthermore, we have $|\{{x_i}^-y_1,y_1y\}\cap A(D)|\leq 1$ and $|\{y_1{x_i}^-,yy_1\}\cap A(D)|\leq 1$, since otherwise there is a longer directed cycle in $D$ or a $C$-path in $D'$ with respect to $H'$, a contradiction. Hence we obtain $d_{D}(y)+d_{D}({x_i}^-)\leq 2|V(D)|-2=2n-2$. Note that $d_{D}(y_1)+d_{D}(s)\leq 2n-2$ by Claim 1 and $y_1$, $y$ are two distinct vertices. Thus we have $d_{D}(y)+d_{D}({x_i}^-)+d_{D}(y_1)+d_{D}(s)\leq 4n-4$, and it contradicts the assumption of Theorem 5.

Assume that there is a $C$-path in $D'$, say $P'=x_{q}z_1\ldots z_{r}x_{q+r}$, where $x_{q},x_{q+r}\in V(C)$. If every vertex in $C[x_{q+1},x_{q+r-1}]$ can be inserted into $C[x_{q+r},x_{q}]$, then by Lemma 3, there is a directed $(x_{q+r},x_q)$-path, say $P_2$, such that $V(P_2)=V(C)$. Then $C'=P_2[x_{q+r},x_q]P'$ is a cycle longer than $C$, contradicting the choice of $C$. Hence there is (are) some vertex (vertices) in $C[x_{q+1},x_{q+r-1}]$ which can not be inserted into $C[x_{q+r},x_{q}]$. W.l.o.g., let $x_{i_1},x_{i_2},\ldots,x_{i_{r'}}$ be such vertex (vertices). By Lemma 1, $d_{D'}(x_{i_j})+d_{D'}(z_1)\leq 2|V(D')|-2$ for any $j\in \{1,2,\ldots,r'\}$. For the vertex $x_{i_1}$, if $|\{x_{i_1}y_1,y_1z_1\}\cap A(D)|\leq 1$ and $|\{z_1y_1,y_1x_{i_1}\}\cap A(D)|\leq 1$, then we obtain $d_{D}(x_{i_1})+d_{D}(z_1)\leq 2|V(D)|-2$. Note that $z_1$ is a vertex different from $y_1$. We have $d(x_{i_1})+d(z_1)+d(s)+d(y_1)\leq 4n-4$, a contradiction. If $z_1y_1,y_1x_{i_1}\in A(D)$, then  note that every vertex of $C(x_q,x_{i_1})$ can be inserted into $C[x_{q+r},x_q]$. Let $C'[x_{q+r},x_q]$ be the resulting path by inserting all vertices of $C(x_q,x_{i_1})$ into $C[x_{q+r},x_q]$. Then $C'=P'[x_q,z_1]z_1y_1x_{i_1}C[x_{i_l},x_{q+r}]C'[x_{q+r},x_q]$ is a longer cycle in $D$, a contradiction. Thus $x_{i_1}y_1,y_1z_1\in A(D)$. By a similar argument as above, we continue this procedure and deduce that $x_{i_{r'}}y_1,y_1z_1\in A(D)$. Now consider the path $P''=x_{i_{r'}}y_1z_1P'[z_1,x_{q+r}]$. Since every vertex in $C(x_{i_{r'}},x_{q+r})$ can be inserted into $C[x_{q+r},x_q]$, we can find a longer cycle in $D$ by a similar argument as above, a contradiction.

This proves this claim.
\end{proof}

By Claim 5, the length of $C$ is $n-1$. The proof is complete. {\hfill$\Box$}

\section{Concluding remarks}
Manousskis \cite{Manoussakis} gave a new type of degree condition for a digraph to be hamiltonian, and it opened up a new area of Hamiltonicity of digraphs for further study. Up to now, there are some results concerning pancyclicity of digraphs with respect to the theorems of Ghouila-Houri, Woodall and Meyniel, respectively. See \cite{Haggkvist_Thomassen,Thomassen}. It is natural to ask whether we can find a similar result for pancyclicity of digraphs under Manoussakis-type degree condition or not. In \cite{Manoussakis}, Manoussakis proposed the following conjecture.

\begin{conjecture}[Manoussakis \cite{Manoussakis}]
Any strongly connected digraph such that for any triple of vertices $x,y,z\in V$, where $x$ is nonadjacent to $y$, there hold $d(x)+d(y)+d^+(x)+d^-(z)\geq 3n+1$ (if $xz\notin A$) and $d(x)+d(y)+d^+(z)+d^-(x)\geq 3n+1$ (if $zx\notin A$) is pancyclic.
\end{conjecture}

Following \cite{Li_Flandrin_Shu}, for a subset $S$ of the vertex set of $D$, we say that $S$ is \emph{cyclable} if there
is a directed cycle in $D$ passing through all vertices of $S$. Berman \& Liu \cite{Berman_Liu} and Li, Flandrin and Shu \cite{Li_Flandrin_Shu} gave
a cyclable version of Meyniel's theorem, independently. Li, Flandrin and Shu \cite{Li_Flandrin_Shu} also proposed the following problem.

\begin{problem}[Li, Flandrin and Shu \cite{Li_Flandrin_Shu}]
Is there a cyclable version of Theorem 4?
\end{problem}

All these problems may stimulate our further study for hamiltonian property of digraphs under Manoussakis-type degree condition.

\section*{Acknowledgements}
The author is supported by NSFC (No. 11271300) and the Doctorate Foundation of Northwestern Polytechnical University (cx201326). He is indebted to Dr. Jun Ge and Dr. Binlong Li for helpful discussions.

\end{document}